\documentclass[12pt,bezier]{article}
\usepackage{times}
\usepackage{booktabs}
\usepackage{pifont}
\usepackage{floatrow}
\usepackage{caption}
\usepackage{mathrsfs}
\usepackage{amsmath}
\usepackage{amsfonts,amsthm,amssymb,mathrsfs,bbding}
\usepackage{txfonts}
\usepackage{graphics,multicol}
\usepackage{graphicx}
\usepackage{color}
\usepackage{amssymb}
\allowdisplaybreaks[3]
\usepackage{caption}
\usepackage{cite}
\usepackage{latexsym,bm}
\usepackage{indentfirst}
\usepackage{color}
\usepackage[colorlinks=true,anchorcolor=blue,filecolor=blue,linkcolor=blue,urlcolor=blue,citecolor=blue]{hyperref}
\usepackage{extarrows}
\usepackage{cite}
\usepackage{latexsym,bm}
\usepackage{mathtools}
\usepackage{enumerate}

\evensidemargin=\oddsidemargin\topmargin=-1.5cm

\newtheorem{thm}{Theorem}[section]

\newtheorem*{claim}{Claim}

\newtheorem{lem}{Lemma}[section]

\theoremstyle{definition}

\addtocounter{section}{0}

\begin{document}
\title{Toughness and distance spectral radius in graphs involving minimum degree\footnote{Supported by the National Natural Science Foundation of China
{(Nos. 12371361, 11971445 and 12171440)}.}}
\author{{\bf Jing Lou$^{a}$}, {\bf Ruifang Liu$^{a}$}\thanks{Corresponding author.
E-mail addresses: rfliu@zzu.edu.cn (R. Liu), loujing\_23@163.com (J. Lou), jlshu@admin.ecnu.edu.cn (J. Shu)},
{\bf Jinlong Shu$^{b}$}\\
{\footnotesize $^a$ School of Mathematics and Statistics, Zhengzhou University, Zhengzhou, Henan 450001, China} \\
{\footnotesize $^b$ School of Finance and Business, Shanghai Normal University, Shanghai 200233, China}}
\date{}

\date{}
\maketitle
{\flushleft\large\bf Abstract}
The {\it toughness} $\tau(G)=\mathrm{min}\{\frac{|S|}{c(G-S)}: S~\mbox{is a cut set of vertices in}~G\}$ for $G\ncong K_n.$ The concept of toughness initially proposed by Chv$\mathrm{\acute{a}}$tal in 1973, which serves as a simple way to measure how tightly various pieces of a graph hold together. A graph $G$ is called {\it $t$-tough} if $\tau(G)\geq t.$ It is very interesting to investigate the relations between toughness and eigenvalues of graphs. Fan, Lin and Lu [European J. Combin. 110 (2023) 103701] provided sufficient conditions in terms of the spectral radius for a graph to be 1-tough with minimum degree $\delta$ and $t$-tough with $t\geq 1$ being an integer, respectively. By using some typical distance spectral techniques and structural analysis, we in this paper present sufficient conditions based on the distance spectral radius to guarantee a graph to be 1-tough with minimum degree $\delta.$ Moreover, we also prove sufficient conditions with respect to the distance spectral radius for a graph to be $t$-tough, where $t$ or $\frac{1}{t}$ is a positive integer.

\begin{flushleft}
\textbf{Keywords:} Toughness, Distance spectral radius, Minimum degree

\end{flushleft}
\textbf{AMS Classification:} 05C50; 05C35

\section{Introduction}
All graphs considered in this paper are undirected and simple. Let $G$ be a graph with vertex set $V(G)$ and edge set $E(G).$ The {\it order} and {\it size} of $G$ are denoted by $|V(G)|=n$ and $|E(G)|=e(G)$, respectively. A graph with just one vertex is referred to as a trivial graph. We denote by $\delta(G)$ and $\Delta(G)$ the minimum degree and the maximum degree of $G,$ respectively. Let $c(G)$ be the number of components of a graph $G$. For a vertex subset $S$ of $G$, we denote by $G-S$ and $G[S]$ the subgraph of $G$ obtained from $G$ by deleting the vertices in $S$ together with their incident edges and the subgraph of $G$ induced by $S$, respectively. As usual, $K_n$ denotes the complete graph of order $n$. For two vertex-disjoint graphs $G_1$ and $G_2$, we denote by $G_{1}+G_{2}$ the {\it disjoint union} of $G_1$ and $G_2$. The {\it join} $G_1\vee G_2$ is the graph obtained from $G_{1}+G_{2}$ by adding all possible edges between $V(G_1)$ and $V(G_2)$. For undefined terms and notions, one can refer to \cite{Bondy2008}.

Let $G$ be a connected graph with vertex set $V(G)=\{v_1, v_2, \ldots, v_n\}$ and edge set $E(G)$. The {\it distance} between $v_i$ and $v_j,$ denote by $d_{ij}(G)$, is the length of a shortest path from $v_i$ to $v_j.$ The {\it distance matrix} of $G,$ denote by $D(G),$ is an $n\times n$ matrix with its rows and columns indexed by $V(G).$ For $i\neq j,$ the $(i,j)$-entry of $D(G)$ is equal to $d_{ij}(G).$ Also, $d_{ii}(G)=0.$ Clearly, $D(G)$ is a real symmetric matrix with zeros on the diagonal. In this paper, we always use $J$ to denote the all-one matrix, $I$ to denote the identity square matrix, and $O$ to denote the zero matrix. We can order the eigenvalues of $D(G)$ as $\lambda_1(D(G))\geq \lambda_2(D(G))\geq \cdots \geq \lambda_n(D(G)).$ By the Perron-Frobenius theorem, $\lambda_1(D(G))$ is always positive (unless $G$ is trivial) and $\lambda_1(D(G))\geq |\lambda_i(D(G))|$ for $i=2, 3, \ldots, n.$ We call $\lambda_1(D(G))$ the {\it distance spectral radius} of $G$. Furthermore, there exists a unique positive unit eigenvector $\mathbf{x}=(x_{1}, x_{2}, \ldots, x_{n})^T$ corresponding to $\lambda_1(D(G))$, which is called the {\it Perron vector} of $D(G).$

Recently, the researchers have paid attention to the problems of establishing relations between structural properties and the distance spectral radius of graphs. In 2021, Zhang and Lin\cite{Zhang2021} presented sufficient conditions in terms of the distance spectral radius to guarantee the existence of a perfect matching in graphs and bipartite graphs, respectively. Subsequently, Zhang, Lin, Liu and Zheng\cite{Zhang2022} generalized the result of \cite{Zhang2021}. Moreover, they also determined the extremal graph attaining the minimum distance spectral radius among all bipartite graphs with a unique perfect matching, and then proved a sufficient condition for the existence of two vertex-disjoint cycles in a bipartite graph with respect to the distance spectral radius. Li and Miao\cite{Li2022} established an upper bound on the distance spectral radius to ensure that a graph has an odd factor. Later, Li, Miao and Zhang\cite{Li2023} presented a sufficient condition based on the distance spectral radius to guarantee the existence of a fractional perfect matching in graphs. Around the same time, Miao and Li\cite{Miao2023} proved an upper bound on the distance spectral radius to ensure that a graph has a star factor. Very recently, Zhang and van Dam\cite{Zhang2023} proposed a sufficient condition based on the distance spectral radius to guarantee that a graph or a bipartite graph is $k$-extendable. Zhou and Wu\cite{Zhou2023} proved an upper bound in terms of the distance spectral radius to ensure the existence of a spanning $k$-tree in graphs.

In 1973, Chv$\mathrm{\acute{a}}$tal\cite{Chvatal1973} initially introduced the concept of toughness, which is regarded as a simple way to measure how tightly various pieces of a graph hold together. The {\it toughness} $\tau(G)=\mathrm{min}\{\frac{|S|}{c(G-S)}: S~\mbox{is a cut set of vertices in}~G\}$ for $G\ncong K_n.$ A graph $G$ is called {\it $t$-tough} if $\tau(G)\geq t.$
Note that $\delta\geq 2$ is a trivial necessary condition for a graph to be 1-tough. In the past few years, many researchers focused on finding sufficient conditions for a graph to be $t$-tough. Very recently, Fan, Lin and Lu\cite{Fan2023} proved a
sufficient condition in terms of the spectral radius for a graph to be 1-tough with minimum degree. Inspired by the work of Fan, Lin and Lu\cite{Fan2023}, we prove a sufficient condition based on the distance spectral radius $\lambda_1(D(G))$ to ensure that a graph $G$ is 1-tough with minimum degree $\delta$.

\begin{thm}\label{main1}
Let $G$ be a connected graph of order $n \geq \{8\delta, \frac{1}{2}\delta^2+2\delta+2\}$ with minimum degree $\delta \geq 2.$ If $$\lambda_1(D(G))\leq \lambda_1(D(K_\delta \vee (K_{n-2\delta}+\delta K_1))),$$
then $G$ is 1-tough unless $G \cong K_\delta \vee (K_{n-2\delta}+\delta K_1)$.
\end{thm}

In the same paper, Fan, Lin and Lu\cite{Fan2023} also proposed a sufficient condition based on the spectral radius to ensure that a graph is $t$-tough. Motivated by their result, we consider the sufficient condition of $t$-tough graphs from the distance spectral radius perspective.

\begin{thm}\label{main2}
Let $G$ be a connected graph of order $n$. Each of the following holds.\\
(i) Let $t$ be a positive integer and $n\geq 4t^2+10t$. If $\lambda_1(D(G))\leq \lambda_1(D(K_{2t-1}\vee(K_{n-2t}+K_1))),$ then $G$ is $t$-tough unless $G\cong K_{2t-1}\vee(K_{n-2t}+K_1)$.\\
(ii) Let $\frac{1}{t}$ be a positive integer and $n\geq 2t+\frac{9}{2 t}+\frac{9}{2}$. If $\lambda_1(D(G))\leq \lambda_1(D(K_{1}\vee(K_{n-1-\frac{1}{t}}+\frac{1}{t}K_1))),$ then $G$ is $t$-tough unless $G\cong K_{1}\vee(K_{n-1-\frac{1}{t}}+\frac{1}{t}K_1)$.
\end{thm}

\section{Preliminary lemmas}
In this section, we put forward some necessary lemmas, which will be used to prove our main results. We first present a preliminary result about the relationship between the distance spectral radius of a graph and its spanning graph, which is a corollary of the Perron-Frobenius theorem.

\begin{lem}[Godsil \cite{Godsil1993}]\label{le1}
Let $e$ be an edge of a graph $G$ such that $G-e$ is connected. Then $$\lambda_1(D(G))<\lambda_1(D(G-e)).$$
\end{lem}

Next we introduce the concepts of equitable partitions and quotient matrices, and then give a well-known result. Let $M$ be a real $n\times n$ matrix. Assume that $M$ can be written as the following matrix
\[
M=\left(\begin{array}{ccccccc}
M_{1,1}&M_{1,2}&\cdots &M_{1,m}\\
M_{2,1}&M_{2,2}&\cdots &M_{2,m}\\
\vdots& \vdots& \ddots& \vdots\\
M_{m,1}&M_{m,2}&\cdots &M_{m,m}\\
\end{array}\right),
\]
whose rows and columns are partitioned into subsets $X_{1}, X_{2},\ldots ,X_{m}$ of $\{1,2,\ldots, n\}$. The quotient matrix $R(M)$ of the matrix $M$ (with respect to the given partition) is the $m\times m$ matrix whose entries are the average row sums of the blocks $M_{i,j}$ of $M$. The above partition is called {\it equitable} if each block $M_{i,j}$ of $M$ has constant row (and column) sum.

\begin{lem}[Brouwer and Haemers \cite{Brouwer2011}, Godsil and Royle \cite{Godsil2001}, Haemers \cite{Haemers1995}]\label{le2}
Let $M$ be a real symmetric matrix and let $R(M)$ be its equitable quotient matrix.
Then the eigenvalues of the quotient matrix $R(M)$ are eigenvalues of $M$.
Furthermore, if $M$ is nonnegative and irreducible, then the spectral radius of the quotient matrix $R(M)$ equals to the spectral radius of $M$.
\end{lem}

Let $W(G)=\sum_{i<j}d_{ij}(G)$ be the {\it Wiener index} of a connected graph $G$ of order $n$. The next result is useful, which can be easily obtained by the Rayleigh quotient\cite{Horn1985}.

\begin{lem}\label{le3}
Let $G$ be a connected graph with order $n$. Then
\begin{equation*}
\lambda_1(D(G))=\mathop{\max}_{\mathbf{x}\neq\mathbf{0}}\frac{\mathbf{x}^TD(G)\mathbf{x}}{\mathbf{x}^T\mathbf{x}}\geq \frac{\mathbf{1}^TD(G)\mathbf{1}}{\mathbf{1}^T\mathbf{1}}=\frac{2W(G)}{n},
\end{equation*}
where $\mathbf{1}=(1, 1, \ldots, 1)^T.$
\end{lem}

Finally, we present an important lemma, which will play an essential role in the proofs of Theorems \ref{main1} and \ref{main2}.

\begin{lem}[Zhang and Lin\cite{Zhang2021}]\label{le4}
Let $n, c, s$ and $n_i~(1\leq i\leq c)$ be positive integers with $n_1\geq n_2\geq \cdots \geq n_c\geq 1$ and $n_1+n_2+\cdots +n_c=n-s.$ Then $$\lambda_1(D(K_s\vee(K_{n_1}+K_{n_2}+\cdots+K_{n_c})))\geq \lambda_1(D(K_s\vee(K_{n-s-(c-1)}+(c-1)K_1)))$$
with equality if and only if $(n_1, n_2, \ldots, n_c)=(n-s-(c-1), 1, \ldots, 1).$
\end{lem}


\section{Proof of Theorem \ref{main1}}

In this section, we give the proof of Theorem \ref{main1}. Before doing this, we need the following critical lemma.

\begin{lem}\label{le6}
Let $n, c, s, p$ and $n_i~(1\leq i\leq c)$ be positive integers with $n_1\geq 2p$, $n_1\geq n_2\geq \cdots \geq n_c\geq p$ and $n_1+n_2+\cdots +n_c=n-s.$
Then $$\lambda_1(D(K_s\vee(K_{n_1}+K_{n_2}+\cdots+K_{n_c})))\geq \lambda_1(D(K_s\vee(K_{n-s-p(c-1)}+(c-1)K_p)))$$
with equality if and only if $(n_1, n_2, \ldots, n_c)=(n-s-p(c-1), p, \ldots, p).$
\end{lem}
\medskip
\noindent  \textbf{Proof.}
Let $G_1=K_s\vee(K_{n-s-p(c-1)}+(c-1)K_p).$ Then the distance matrix $D(G_1)$ of $G_1$ is as follows
\[
\bordermatrix{
                  & s    & n-s-(c-1)p  & p    &\cdots& p \cr
\hfill s          & J-I  &  J   & J    &\cdots& J   \cr
\hfill n-s-(c-1)p & J    & J-I  & 2J   &\cdots& 2J  \cr
\hfill p          & J    & 2J   & J-I  &\cdots& 2J  \cr
\hfill \vdots     &\vdots&\vdots&\vdots&\ddots&\vdots\cr
\hfill p          & J    &2J    & 2J   &\cdots&J-I
}.
\]
Let $\mathbf{x}$ be the Perron vector of $D(G_1)$. By symmetry, we take $x_u=x_1$ for all $u\in V(K_s),$ $x_v=x_2$ for all $v\in V(K_{n-s-p(c-1)}),$ and $x_w=x_3$ for all $w\in V((c-1)K_p).$ According to $D(G_1)\mathbf{x}=\lambda_1(D(G_1))\mathbf{x},$ we have
\begin{equation*}
\left\{
\begin{array}{c}
\lambda_1(D(G_1))x_2=sx_1+[n-s-(c-1)p-1]x_2+2(c-1)px_3,\hfill\\
\lambda_1(D(G_1))x_3=sx_1+2[n-s-(c-1)p]x_2+[(p-1)+2(c-2)p]x_3,\hfill
\end{array}
\right.
\end{equation*}
which leads to $$\lambda_1(D(G_1))(x_3-x_2)=[n-s-(c-1)p+1]x_2-(p+1)x_3.$$
It follows that  $$(\lambda_1(D(G_1))+p+1)x_3=[\lambda_1(D(G_1))+n-s-(c-1)p+1]x_2.$$
Note that $\lambda_1(D(G_1))>0$ and $p\geq 1$. Then
\begin{equation}\label{eq1}
x_3=\frac{\lambda_1(D(G_1))+n-s-(c-1)p+1}{\lambda_1(D(G_1))+p+1}x_2=1+\frac{n-s-cp}{\lambda_1(D(G_1))+p+1}x_2.
\end{equation}

Let $G_2=K_s\vee(K_{n_1}+K_{n_2}+\cdots+K_{n_c}).$ We can write its distance matrix $D(G_2)$ as follows
\[
\bordermatrix{
              & s    & n_1  & n_2  &\cdots& n_c \cr
\hfill s      & J-I  &  J   & J    &\cdots& J   \cr
\hfill n_1    & J    & J-I  & 2J   &\cdots& 2J  \cr
\hfill n_2    & J    & 2J   & J-I  &\cdots& 2J  \cr
\hfill\vdots  &\vdots&\vdots&\vdots&\ddots&\vdots\cr
\hfill n_c    & J    &2J    & 2J   &\cdots&J-I
}.
\]
Let $\mathbf{x}=(\underbrace{x_1, \ldots, x_1}_{s}, \underbrace{x_3, \ldots, x_3}_{(c-1)p}, \underbrace{x_2, \ldots, x_2}_{n_1}, \underbrace{x_2, \ldots, x_2}_{n_2-p}, \underbrace{x_2, \ldots, x_2}_{n_3-p}, \ldots, \underbrace{x_2, \ldots, x_2}_{n_c-p})^T.$ Then the corresponding matrix $D(G_2)-D(G_1)$ is as follows
\[
\bordermatrix{
               & s     & (c-1)p     & n_1   & n_2-p & n_3-p  &\cdots & n_c-p    \cr
\hfill s       & O     & O          & O     & O     & O      &\cdots & O        \cr
\hfill (c-1)p  & O     & O          & O     & -E_1  & -E_2   &\cdots & -E_{c-1} \cr
\hfill n_1     & O     & O          & O     & J     & J      &\cdots & J        \cr
\hfill n_2-p   & O     & -E_1^T     & J     & O     & J      &\cdots & J        \cr
\hfill n_3-p   & O     & -E_2^T     & J     & J     & O      &\cdots & J        \cr
\hfill \vdots  &\vdots &\vdots      &\vdots &\vdots &\vdots  &\ddots &\vdots    \cr
\hfill n_c-p   & O     & -E_{c-1}^T & J     & J     & J      &\cdots & O
},
\]
where $E_i$ denotes the matrix whose each entry from $[(i-1)p+1]$-th row to $ip$-th row is 1 and whose other entries are 0. Then
\begin{eqnarray*}
&&\lambda_1(D(G_2))-\lambda_1(D(G_1))\geq\mathbf{x}^{T}(D(G_2)-D(G_1))\mathbf{x}\\
&=& -(n_2-p)px_2x_3-(n_3-p)px_2x_3-\cdots-(n_c-p)px_2x_3+n_1\sum_{i=2}^{c}(n_i-p)x_2^2\\
&&+x_2(n_2-p)\Bigg[-px_3+n_1x_2+\sum_{i=2,i\neq2}^{c}(n_i-p)x_2\Bigg]+x_2(n_3-p)\Bigg[-px_3+n_1x_2\Bigg.\\
&&\Bigg.+\sum_{i=2,i\neq3}^{c}(n_i-p)x_2\Bigg]+\cdots+x_2(n_c-p)\Bigg[-px_3+n_1x_2+\sum_{i=2,i\neq c}^{c}(n_i-p)x_2\Bigg]\\
&=&n_1\sum_{i=2}^{c}(n_i-p)x_2^2+(n_2-p)[(n-s-n_2-(c-2)p)x_2^2-2px_2x_3]\\
&&+(n_3-p)[(n-s-n_3-(c-2)p)x_2^2-2px_2x_3]\\
&&+\cdots+(n_c-p)[(n-s-n_c-(c-2)p)x_2^2-2px_2x_3].
\end{eqnarray*}
Note that $G_1$ contains $K_{n-(c-1)p}$ as a proper subgraph. Then $\lambda_1(D(G_1))>\lambda_1(D(K_{n-(c-1)p}))=n-(c-1)p-1.$ Combining this with (\ref{eq1}), $p\geq 1, s\geq 1$ and $n_1\geq n_2\geq\cdots\geq n_c\geq p$, for any $k=2, 3,\ldots, c,$ we have
\begin{eqnarray*}
 && (n-s-n_k-(c-2)p)x_2^2-2px_2x_3\\
&=& x_2^2\left[n-s-n_k-(c-2)p-2p\left(1+\frac{n-s-cp}{\lambda_1(D(G_1))+p+1}\right)\right]\\
&=& x_2^2\left[n-s-n_k-cp-\frac{2p(n-s-cp)}{\lambda_1(D(G_1))+p+1}\right]\\
&>& x_2^2\left[n-s-n_k-cp-\frac{2p(n-s-cp)}{n-cp+2p}\right]\\
&=& x_2^2\left(n-s-n_k-cp-2p+\frac{4p^2+2sp}{n-cp+2p}\right)\\
&>& x_2^2(n-s-n_k-cp-2p)\\
&=& x_2^2\left(\sum_{i=1,i\neq k}^{c}n_i-cp-2p\right)\\
&\geq& x_2^2[n_1+(c-2)p-cp-2p]\\
&=& x_2^2(n_1-4p).
\end{eqnarray*}
Since $n_1\geq 2p$ and $n_2\geq n_3\geq\cdots\geq n_c\geq p,$  we have
\begin{eqnarray}\label{eq2}
&& \lambda_1(D(G_2))-\lambda_1(D(G_1))\nonumber\\
&\geq& n_1\sum_{i=2}^{c}(n_i-p)x_2^2+(n_2-p)[(n-s-n_2-(c-2)p)x_2^2-2px_2x_3]\nonumber\\
&&+(n_3-p)[(n-s-n_3-(c-2)p)x_2^2-2px_2x_3]\nonumber\\
&&+\cdots+(n_c-p)[(n-s-n_c-(c-2)p)x_2^2-2px_2x_3]\nonumber\\
&\geq& n_1\sum_{i=2}^{c}(n_i-p)x_2^2+x_2^2(n_1-4p)\sum_{i=2}^{c}(n_i-p)\nonumber\\
&=& x_2^2(2n_1-4p)\sum_{i=2}^{c}(n_i-p)\nonumber\\
&\geq& 0.
\end{eqnarray}
Hence $\lambda_1(D(G_2))\geq\lambda_1(D(G_1)).$

If $(n_1, n_2, \ldots, n_c)=(n-s-p(c-1), p, \ldots, p),$ then $K_s\vee(K_{n_1}+K_{n_2}+\cdots+K_{n_c})\cong K_s\vee(K_{n-s-p(c-1)}+(c-1)K_p),$ and we have $\lambda_1(D(G_1))= \lambda_1(D(G_2)).$ Conversely, if $\lambda_1(D(G_1))= \lambda_1(D(G_2)),$ then all the inequalities in (\ref{eq2}) must be equalities, and hence $n_2=n_3=\cdots=n_c=p.$ So $(n_1, n_2, \ldots, n_c)=(n-s-p(c-1), p, \ldots, p).$
\hspace*{\fill}$\Box$

\medskip
Now, we are in a position to present the proof of Theorem \ref{main1}.

\medskip
\noindent  \textbf{Proof of Theorem \ref{main1}.}
Let $G$ be a connected graph of order $n\geq \mbox{max}\{8\delta, \frac{1}{2}\delta^2+2\delta+2\}$ and minimum degree $\delta\geq 2.$ Suppose to the contrary that $G$ is not 1-tough. By the definition of 1-tough graphs, then $0<\tau(G)<1$, and hence there exists a vertex subset $S\subseteq V(G)$ such that $c(G-S)>|S|.$ Let $|S|=s$ and $c(G-S)=c$. Then $c\geq s+1,$ and hence $n\geq 2s+1.$ It is obvious that $G$ is a spanning subgraph of $G'=K_s\vee(K_{n_1}+K_{n_2}+\cdots+K_{n_{s+1}})$ for some integers $n_1\geq n_2\geq\cdots\geq n_{s+1}\geq 1$ and $\sum_{i=1}^{s+1}n_i=n-s.$ According to Lemma \ref{le1}, we have
\begin{equation}\label{eq3}
\lambda_1(D(G'))\leq \lambda_1(D(G)),
\end{equation}
where equality holds if and only if $G\cong G'$. Note that $s\geq 1$. Next we divide the proof into the following three cases.

\vspace{1.8mm}
\noindent\textbf{Case 1.} $s\geq\delta+1.$
\vspace{1mm}

Let $\tilde{G}=K_s\vee(K_{n-2s}+sK_1).$ By Lemma \ref{le4}, we have
\begin{equation}\label{eq4}
\lambda_1(D(\tilde{G}))\leq\lambda_1(D(G')),
\end{equation}
with equality holding if and only if $G'\cong \tilde{G}.$
Define $G^*=K_\delta\vee(K_{n-2\delta}+\delta K_1).$ Then its distance matrix $D(G^*)$ is
\[
\bordermatrix{
                  & \delta    & n-2\delta  & \delta \cr
\hfill \delta     & 2(J-I)    &  2J        & J      \cr
\hfill n-2\delta  & 2J        & J-I        & J      \cr
\hfill \delta     & J         & J          & J-I    \cr
}.
\]
We can partition the vertex set of $G^*$ as $V(G^*)=V(\delta K_{1})\cup V(K_{n-2\delta})\cup V(K_{\delta}).$ Then the quotient matrix of $D(G^*)$ with respect to this partition is
$$R_{\delta}=\left(
\begin{array}{ccc}
2(\delta-1) &2(n-2\delta) & \delta\\
2\delta     & n-2\delta-1 & \delta\\
\delta      & n-2\delta   & \delta-1
\end{array}
\right).
$$
By simple calculation, the characteristic polynomial of $R_{\delta}$ is
\begin{equation}\label{eq5}
P(R_\delta,x)=x^3-(n+a-4)x^2-(2an+3n-5a^2+a-5)x+a^2n-2an-2n-2a^3+5a^2+2.
\end{equation}
Note that the above partition is equitable. By Lemma \ref{le2}, $\lambda_1(D(G^*))=\lambda_1(R_\delta)$ is the largest root of the equation $P(R_\delta,x)=0.$ Recall that $\tilde{G}=K_{s}\vee (K_{n-2s}+sK_1)$. Observe that $D(\tilde{G})$ has the equitable quotient matrix $R_s$, which is obtained by replacing $\delta$ with $s$ in $R_\delta.$ Similarly, by Lemma \ref{le2}, $\lambda_1(D(\tilde{G}))=\lambda_1(R_s)$ is the largest root of the equation $P(R_s,x)=0.$ Then
$$P(R_\delta,x)-P(R_s,x)=(s-\delta)[x^2+(2n-5s-5\delta+1)x-sn-\delta n+2n+2s^2+2\delta s-5s+2\delta^2-5\delta].$$
\begin{claim}
$P(R_\delta,x)-P(R_s,x)>0$ for $x\in [n+\delta, +\infty)$.
\end{claim}

\begin{proof} Define $f(x)=x^2+(2n-5s-5\delta+1)x-sn-\delta n+2n+2s^2+2\delta s-5s+2\delta^2-5\delta.$
Note that $s\geq\delta+1.$ It suffices to prove that $f(x)>0$ for $x\in [n+\delta, +\infty).$ Recall that $n\geq 2s+1$. Then $\delta+1\leq s \leq \frac{n-1}{2},$ and hence the symmetry axis of $f(x)$
\begin{eqnarray*}
x &=&-n+\frac{5}{2}s+\frac{5}{2}\delta-\frac{1}{2}\\
  &=&(n+\delta)-2n+\frac{5}{2}s+\frac{3}{2}\delta-\frac{1}{2}\\
  &\leq&(n+\delta)-2(2s+1)+\frac{5}{2}s+\frac{3}{2}\delta-\frac{1}{2}\\
  &=&(n+\delta)-\frac{3}{2}s+\frac{3}{2}\delta-\frac{5}{2}\\
  &\leq&(n+\delta)-\frac{3}{2}(\delta+1)+\frac{3}{2}\delta-\frac{5}{2}\\
  &=&(n+\delta)-4\\
  &<&n+\delta.
\end{eqnarray*}
This implies that $f(x)$ is monotonically increasing with respect to $x\in [n+\delta, +\infty)$. Since $\delta+1\leq s\leq \frac{n-1}{2}$, $n\geq 8\delta$ and $\delta \geq 2$, we have
\begin{eqnarray*}
f(x)&\geq& f(n+\delta)\\
&=&2s^2-(6n+3\delta+5)s+3n^2-2\delta n+3n-2\delta^2-4\delta\\
&\geq&2(\frac{n-1}{2})^2-(6n+3\delta+5)(\frac{n-1}{2})+3n^2-2\delta n+3n-2\delta^2-4\delta\\
&=& \frac{1}{2}n^2-(\frac{7}{2}\delta-\frac{5}{2})n-2\delta^2-\frac{5}{2}\delta+3\\
&\geq& \frac{1}{2}(8\delta)^2-(\frac{7}{2}\delta-\frac{5}{2})(8\delta)-2\delta^2-\frac{5}{2}\delta+3\\
&=& 2\delta^2+\frac{35}{2}\delta+3\\
&>& 0.
\end{eqnarray*}
It follows that $P(R_{\delta},x)>P(R_{s},x)$ for $x\geq n+\delta.$
\end{proof}
\noindent Note that
\begin{eqnarray*}
W(G^*)&=&\sum_{i<j}d_{ij}(G^*)\\
&=&\frac{2[1+(\delta-1)](\delta-1)}{2}+2(n-2\delta)\delta+\delta^2+\frac{[1+(n-\delta-1)](n-\delta-1)}{2}\\
&=&\frac{1}{2}n^2+(\delta-\frac{1}{2})n-\frac{3}{2}\delta^2-\frac{1}{2}\delta.
\end{eqnarray*}
By Lemma \ref{le3}, $n\geq 8\delta$ and $\delta \geq 2,$ we have
\begin{eqnarray*}
\lambda_1(D(G^*))\geq\frac{2W(G^*)}{n}&=&\frac{n^2+(2\delta-1)n-3\delta^2-\delta}{n}\\
&=&n+2\delta-1-\frac{3\delta^2+\delta}{n}\\
&\geq&(n+\delta)+\delta-1-\frac{3\delta^2+\delta}{8\delta}\\
&=&(n+\delta)+\frac{5\delta^2-9\delta}{8\delta}\\
&>&n+\delta.
\end{eqnarray*}
Combining $P(R_{\delta},x)>P(R_{s},x)$ for $x\in [n+\delta, +\infty)$ and $\lambda_1(D(G^*))>n+\delta,$ we obtain that $\lambda_1(D(G^*))<\lambda_1(D(\tilde{G})).$
By (\ref{eq3}) and (\ref{eq4}), we have
$$\lambda_1(D(G^*))<\lambda_1(D(\tilde{G}))\leq\lambda_1(D(G'))\leq\lambda_1(D(G)),$$
which contradicts the assumption.

\vspace{1.8mm}
\noindent\textbf{Case 2.} $s=\delta$.
\vspace{1mm}

Then $G'\cong K_\delta\vee(K_{n_1}+K_{n_2}+\cdots+K_{n_{\delta+1}})$. By Lemma \ref{le4}, we have
\begin{equation}\label{eq6}
\lambda_1(D(K_\delta\vee(K_{n-2\delta}+\delta K_1)))\leq\lambda_1(D(G')),
\end{equation}
with equality holding if and only if $G'\cong K_\delta\vee(K_{n-2\delta}+\delta K_1).$
By (\ref{eq3}) and (\ref{eq6}), we have $$\lambda_1(D(K_\delta\vee(K_{n-2\delta}+\delta K_1)))\leq\lambda_1(D(G)),$$
where equality holds if and only $G\cong K_\delta\vee(K_{n-2\delta}+\delta K_1).$
By the assumption $\lambda_1(D(G))\leq \lambda_1(D(K_\delta\vee(K_{n-2\delta}+\delta K_1))),$ we have $\lambda_1(D(K_\delta\vee(K_{n-2\delta}+\delta K_1)))=\lambda_1(D(G)),$ and hence $G\cong K_\delta\vee(K_{n-2\delta}+\delta K_1)$ (see Fig. \ref{fig1}). Take $S=V(K_\delta).$ Then $$\frac{|S|}{c(K_\delta\vee(K_{n-2\delta}+\delta K_1)-S)}=\frac{\delta}{\delta+1}<1,$$ and hence $\tau(K_\delta\vee(K_{n-2\delta}+\delta K_1)<1.$ This implies that $K_\delta\vee(K_{n-2\delta}+\delta K_1)$ is not 1-tough. So $G'\cong K_\delta\vee(K_{n-2\delta}+\delta K_1).$

\begin{figure}
\centering
{\includegraphics[width=0.35\textwidth]{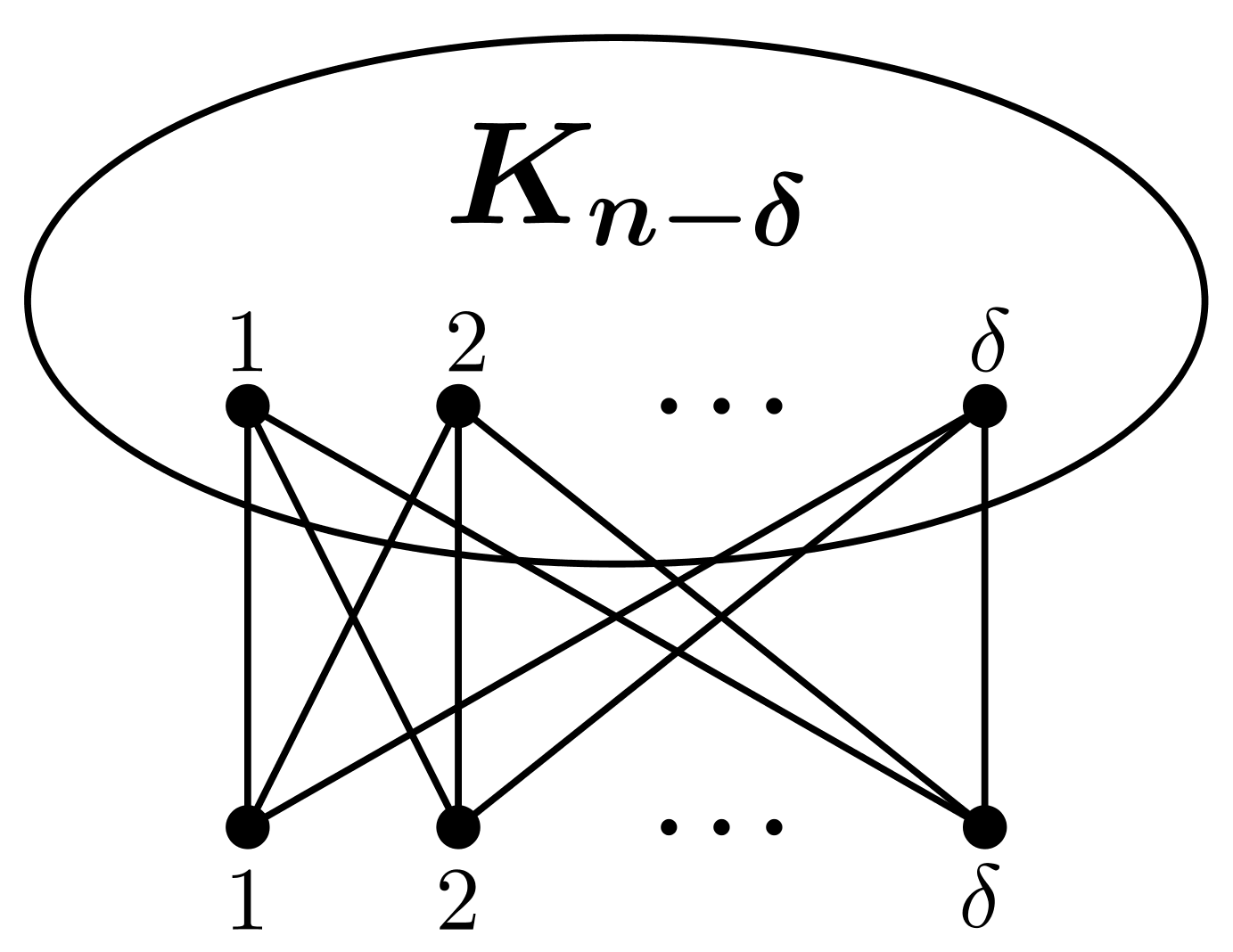}}\\
\caption{Graph $K_{\delta}\vee (K_{n-2\delta}+\delta K_{1}).$}
\label{fig1}
\end{figure}

\vspace{1.8mm}
\noindent\textbf{Case 3.} $1\leq s<\delta.$
\vspace{1mm}

Recall that $G$ is a spanning subgraph of $G'=K_s\vee(K_{n_1}+K_{n_2}+\cdots+K_{n_{s+1}}),$ where $n_1\geq n_2\geq\cdots\geq n_{s+1}$ and $\sum_{i=1}^{s+1}n_i=n-s.$ Note that $\delta(G')\geq\delta(G)=\delta.$ Then $n_{s+1}-1+s\geq\delta.$ Hence $n_1\geq n_2\geq\cdots\geq n_{s+1}\geq \delta-s+1.$ We assert that $n_1\geq 2(\delta-s+1).$ In fact, if $n_1\leq 2\delta-2s+1.$ Notice that $n_1\geq n_2\geq\cdots\geq n_{s+1}$ and $1\leq s \leq \delta-1.$ Then we have
\begin{eqnarray*}
n&=& s+n_1+n_2+\cdots+n_{s+1}\\
&\leq& s+(s+1)(2\delta-2s+1)\\
&=&-2s^2+2\delta s+2\delta+1\\
&\leq& -2(\frac{1}{2}\delta)^2+2\delta(\frac{1}{2}\delta)+2\delta+1\\
&=&\frac{1}{2}\delta^2+2\delta+1,
\end{eqnarray*}
which contradicts $n\geq \frac{1}{2}\delta^2+2\delta+2.$ Let $\hat{G}=K_s\vee(K_{n-s-(\delta-s+1)s}+sK_{\delta-s+1}).$ By Lemma \ref{le6}, we have
\begin{equation}\label{eq7}
\lambda_1(D(\hat{G}))\leq\lambda_1(D(G')),
\end{equation}
where equality holds if and only if $G'\cong \hat{G}.$ Next we divide the proof into two cases.

\vspace{1.8mm}
\noindent\textbf{Case 3.1.} $s=1.$
\vspace{1mm}

Then $\hat{G}=K_1\vee(K_{n-\delta-1}+K_\delta),$ and its distance matrix is
\[
\bordermatrix{
                  & \delta    & n-\delta-1 & 1 \cr
\hfill \delta     & J-I       &  2J        & J      \cr
\hfill n-\delta-1 & 2J        & J-I        & J      \cr
\hfill 1          & J         & J          & O    \cr
}.
\]
Recall that $G^*=K_\delta\vee(K_{n-2\delta}+\delta K_1).$ Let $\mathbf{x}$ be the Perron vector of $D(G^*)$. By symmetry, $\mathbf{x}$ takes the same values on the vertices of $V(\delta K_1), V(K_{n-2\delta})$ and $V(K_\delta),$ respectively. We denote the entry of $\mathbf{x}$ by $x_1, x_2$ and $x_3$ corresponding to the vertices in the above three vertex sets, respectively. By $D(G^*)\mathbf{x}=\lambda_1(D(G^*))\mathbf{x},$ we have
\begin{equation*}
\left\{
\begin{array}{c}
\lambda_1(D(G^*))x_1=2(\delta-1)x_1+2(n-2\delta)x_2+\delta x_3,\hfill\\
\lambda_1(D(G^*))x_3=\delta x_1+(n-2\delta)x_2+(\delta-1) x_3,\hfill
\end{array}
\right.
\end{equation*}
which leads to $$\lambda_1(D(G^*))(2x_3-x_1)=2x_1+(\delta-2)x_3.$$
It follows that  $$(\lambda_1(D(G^*))+1)(2x_3-x_1)=x_1+\delta x_3.$$ Note that $x_1, x_3>0$ and $\lambda_1(D(G^*))>0.$ Then $2x_3>x_1.$ By direct calculation, we obtain that the matrix $D(\hat{G})-D(G^*)$ is
\[
\bordermatrix{
               & \delta    & n-2\delta  & \delta-1 &1 \cr
\hfill \delta      & -(J-I) & O         & J        &O \cr
\hfill n-2\delta   & O      & O         & O        &O \cr
\hfill \delta-1    & J      & O         & O        &O \cr
\hfill 1           & O      & O         & O        &O
}.
\]
Therefore, we have
\begin{eqnarray*}
\lambda_1(D(\hat{G}))-\lambda_1(D(G^*))&\geq&\mathbf{x}^{T}(D(\hat{G})-D(G^*))\mathbf{x}\\
&=&-\delta(\delta-1)x_1^2+2\delta(\delta-1)x_1x_3\\
&=&\delta(\delta-1)(2x_3-x_1)x_1\\
&>&0.
\end{eqnarray*}
Hence $\lambda_1(D(G^*))<\lambda_1(D(\hat{G})).$

\vspace{1.8mm}
\noindent\textbf{Case 3.2.} $2\leq s<\delta.$
\vspace{1mm}

Recall that $\hat{G}=K_s\vee(K_{n-s-(\delta-s+1)s)}+sK_{\delta-s+1}).$ The distance matrix $D(\hat{G})$ of $\hat{G}$ becomes
\[
\bordermatrix{
                          &\delta-s+1&\cdots &\delta-s+1 &n-s-(\delta-s+1)s  & s \cr
\hfill \delta-s+1         & J-I      &\cdots &2J         &2J                 & J \cr
\hfill \vdots             & 2J       &\cdots &2J         &2J                 & J \cr
\hfill \delta-s+1         & 2J       &\cdots &J-I        &2J                 & J \cr
\hfill n-s-(\delta-s+1)s  & 2J       &\cdots &2J         &J-I                & J \cr
\hfill s                  & J        &\cdots &J          &J                  & J-I
}.
\]
We can partition the vertex set of $\hat{G}$ as $V(\hat{G})=V(s K_{\delta-s+1})\cup V(K_{n-s-(\delta-s+1)s})\cup V(K_{s}).$ Then the quotient matrix of $\hat{G}$ with respect to the partition is
$$R_{s,\delta}=\left(
\begin{array}{ccc}
(\delta-s)+2(s-1)(\delta-s+1) &2[n-s-(\delta-s+1)s] & s\\
2s(\delta-s+1)                & n-s-(\delta-s+1)s-1 & s\\
s(\delta-s+1)                 & n-s-(\delta-s+1)s   & s-1
\end{array}
\right).
$$
By direct calculation, the characteristic polynomial of $R_{s,\delta}$ is
\begin{eqnarray}\label{eq8}
P(R_{s,\delta},x)&=&x^3+[s^2-(\delta+2)s+\delta-n+4]x^2+[2s^4-(4\delta+6)s^3+(2\delta^2+5\delta+2n\nonumber\\
&&+5)s^2+(\delta^2-2n\delta-n-3)s-n\delta+2\delta-3n+5]x-s^5+(2\delta+5)s^4\nonumber\\
&&-(\delta^2+7\delta+n+8)s^3+(2\delta^2+n\delta+5\delta+3n+4)s^2+(\delta^2-2n\delta+\delta-n-1)s\nonumber\\
&&-n\delta+\delta-2n+2.
\end{eqnarray}
Note that the partition is equitable. By Lemma \ref{le2}, $\lambda_1(D(\hat{G}))=\lambda_1(R_{s,\delta})$ is the largest root of the equation $P(R_{s,\delta},x)=0.$ Note that $G^*$ contains $K_{n-\delta}$ as a proper subgraph. Then
$\lambda_1(D(G^*))>\lambda_1(D(K_{n-\delta}))=n-\delta-1.$
Combining (\ref{eq5}) and (\ref{eq8}), we obtain that
\begin{eqnarray*}
&&P(R_\delta,n-\delta-1)-P(R_{s,\delta},n-\delta-1)\\
&=&(\delta-s)[3(s-1)n^2+(2s^3-(2\delta+7)s^2-(5\delta-4)s+11\delta+1)n-s^4-(\delta-3)s^3\\
&&+(2\delta^2+6\delta-2)s^2+(2\delta^2-5\delta)s-9\delta^2-\delta]\\
&\triangleq&(\delta-s)g(n).
\end{eqnarray*}
Note that $2\leq s\leq \delta-1$ and $\delta\geq s+1\geq 3$. Hence the symmetry axis of $g(n)$ is
\begin{eqnarray*}
n&=&\frac{-2s^3+(2\delta+7)s^2+(5\delta-4)s-11\delta-1}{6(s-1)}\\
&=&\frac{-2(s-1)^3+(2\delta+1)(s-1)^2+(9\delta+4)(s-1)-4\delta}{6(s-1)}\\
&=&-\frac{1}{3}(s-1)^2+\frac{2\delta+1}{6}(s-1)+\frac{9\delta+4}{6}-\frac{4\delta}{6(s-1)}\\
&<&-\frac{1}{3}(s-1)^2+\frac{2\delta+1}{6}(s-1)+\frac{9\delta+4}{6}.\\
\end{eqnarray*}
If $3\leq\delta\leq\frac{9}{2},$ then
\begin{eqnarray*}
n&<&-\frac{1}{3}(s-1)^2+\frac{2\delta+1}{6}(s-1)+\frac{9\delta+4}{6}\\
&\leq&-\frac{1}{3}(\delta-2)^2+\frac{2\delta+1}{6}(\delta-2)+\frac{9\delta+4}{6}\\
&=&\frac{7}{3}\delta-1<\frac{1}{2}\delta^2+2\delta+2.
\end{eqnarray*}
If $\delta> \frac{9}{2},$ then
\begin{eqnarray*}
n&<&-\frac{1}{3}(s-1)^2+\frac{2\delta+1}{6}(s-1)+\frac{9\delta+4}{6}\\
&\leq&-\frac{1}{3}(\frac{1}{2}\delta+\frac{1}{4})^2+\frac{2\delta+1}{6}(\frac{1}{2}\delta+\frac{1}{4})+\frac{9\delta+4}{6}\\
&=&\frac{1}{12}\delta^2+\frac{19}{12}\delta+\frac{11}{16}<\frac{1}{2}\delta^2+2\delta+2.
\end{eqnarray*}
This implies that $g(n)$ is monotonically increasing with respect to $n\in [\frac{1}{2}\delta^2+2\delta+2, +\infty)$.
Note that $s\geq 2$ and $\delta\geq s+1\geq3.$ Then we have
\begin{eqnarray*}
g(n)&\geq&g(\frac{1}{2}\delta^2+2\delta+2)\\
&=&\frac{\delta}{4}[\delta(3(s-1)\delta^2-(4s^2-14s+2)\delta+4s^3-22s^2+48s-18)+12s^3-48s^2\\
&&+68s-4]-(s-1)(s^3-6s^2+10s-10)\\
&\geq&\frac{\delta}{4}[\delta(3s^3-9s^2+57s-23)+12s^3-48s^2+68s-4]-(s-1)(s^3-6s^2\\
&&+10s-10)\\
&\geq&\frac{\delta}{4}(3s^4+6s^3+102s-27)-(s-1)(s^3-6s^2+10s-10)\\
&\geq&\frac{1}{4}(3s^5+5s^4+34s^3+38s^2+115s-67)\\
&>&0.
\end{eqnarray*}
Since $\delta\geq s+1,$ we have
\begin{equation}\label{eq9}
P(R_\delta,n-\delta-1)>P(R_{s,\delta},n-\delta-1).
\end{equation}
For $x\in [n-\delta-1,+\infty)$ and $s\geq2,$ we have
\begin{eqnarray*}
&&P'(R_\delta,x)-P'(R_{s,\delta},x)\\
&=&(\delta-s)[(2s-4)x+2s^3-(2\delta+6)s^2-(\delta-2n-5)s+5\delta-n-3]\\
&\geq&(\delta-s)[(2s-4)(n-\delta-1)+2s^3-(2\delta+6)s^2-(\delta-2n-5)s+5\delta-n-3]\\
&=&(\delta-s)[2s^3-(2\delta+6)s^2+(4n-3\delta+3)s+9\delta-5n+1]\\
&\triangleq&(\delta-s)h(s).
\end{eqnarray*}
Next we prove that $h(s)>0$ for $2\leq s\leq \delta-1.$ By direct calculation, we deduce that
$$h'(s)=6s^2-4(\delta+3)s+4n-3\delta+3,$$
and the  symmetry axis of $h'(s)$ is $s=\frac{1}{3}\delta+1.$ Since $n\geq\frac{1}{2}\delta^2+2\delta+2$ and $\delta\geq 3,$ we have
$$h'(s)\geq h'(\frac{1}{3}\delta+1)=4n-\frac{2}{3}\delta^2-7\delta-3\geq\frac{4}{3}\delta^2+\delta+5>0.$$
It follows that $h(s)$ is monotonically increasing for $2\leq s\leq \delta-1.$ Combining this with $n\geq 8\delta$ and $\delta\geq 3,$ we obtain that
$$h(s)\geq h(2)=3n-5\delta-1\geq 19\delta-1>0.$$
Note that $\delta\geq s+1.$ Then
\begin{equation}\label{eq10}
P'(R_\delta,x)>P'(R_{s,\delta},x).
\end{equation}
Moreover, we consider $P'(R_\delta,x)=3x^2-2(\delta+n-4)x+5\delta^2-2n\delta-\delta-3n+5.$
Note that $n\geq 8\delta$ and $\delta\geq 3.$ Then the symmetry axis of $P'(R_\delta,x)$ is
\begin{eqnarray*}
x=\frac{\delta+n-4}{3}&=&(n-\delta-1)-\frac{2}{3}n+\frac{4}{3}\delta-\frac{1}{3}\\
&\leq&(n-\delta-1)-\frac{2}{3}(8\delta)+\frac{4}{3}\delta-\frac{1}{3}\\
&=&(n-\delta-1)-4\delta-\frac{1}{3}\\
&<&n-\delta-1.
\end{eqnarray*}
Then we have $$P'(R_\delta,x)\geq P'(R_\delta,n-\delta-1)=n^2-(8\delta-1)n+10\delta^2-\delta\geq10\delta^2+7\delta>0.$$
It follows that $P(R_\delta,x)$ is monotonically increasing with respect to $x\in [n-\delta-1, +\infty).$ Combining this with (\ref{eq9}) and (\ref{eq10}), we deduce that $\lambda_1(D(G^*))<\lambda_1(D(\hat{G})).$

By (\ref{eq3}), (\ref{eq7}), Case 3.1 and Case 3.2, we have $$\lambda_1(D(G^*))<\lambda_1(D(\hat{G}))\leq\lambda_1(D(G'))\leq\lambda_1(D(G)),$$
which contradicts the assumption.
\hspace*{\fill}$\Box$

\section{Proof of Theorem \ref{main2}}

Let $W^{(2)}(G)$ denote the sum of the squares of the distances between all unordered pairs of vertices in the graph. That is to say,
$$W^{(2)}(G)=\sum_{1\leq i<j\leq n}d_{ij}^2(G).$$

\begin{lem}[Zhou and Trinajsti$\mathrm{\acute{c}}$ \cite{Zhou2010,Zhou2007,Zhou2007(2)}]\label{le7}
Let $G$ be a graph on $n\geq 2$ vertices with sum of the squares of the distances between all unordered pairs of vertices $W^{(2)}(G).$ Then
$$\lambda_1(D(G))\leq \sqrt{\frac{2(n-1)W^{(2)}(G)}{n}}$$
with equality if and only if $G$ is the complete graph $K_n$, and if $G$ has exactly one positive distance eigenvalue, then
$$\lambda_1(D(G))\geq \sqrt{W^{(2)}(G)}$$
with equality if and only if $G$ is $K_2.$
\end{lem}

Now we are ready to present the proof of Theorem \ref{main2}.

\medskip
\noindent  \textbf{Proof of Theorem \ref{main2}.}
Let $G$ be a connected graph of order $n.$
Assume to the contrary that $G$ is not $t$-tough. Then $0<\tau(G)<t,$ and hence there exists a vertex subset $S\subseteq V(G)$ such that $tc(G-S)>|S|.$ Let $|S|=s$ and $c(G-S)=c$. Then $tc>s.$

(i) When $t$ is a positive integer, we have $tc\geq s+1.$ Note that $G$ is a spanning subgraph of $G'=K_{tc-1}\vee(K_{n_1}+K_{n_2}+\cdots+K_{n_{c}}),$ where $n_1\geq n_2\geq\cdots\geq n_{c}\geq 1$ and $\sum_{i=1}^{c}n_i=n-tc+1.$ By Lemma \ref{le1}, we have
\begin{equation}\label{eq11}
\lambda_1(D(G'))\leq \lambda_1(D(G)),
\end{equation}
where equality holds if and only if $G\cong G'$.
Let $G''=K_{tc-1}\vee(K_{n-(t+1)c+2}+(c-1)K_1).$ By Lemma \ref{le4}, we have
\begin{equation}\label{eq12}
\lambda_1(D(G''))\leq \lambda_1(D(G')),
\end{equation}
with equality holding if and only if $G'\cong G''.$
Next we divide the proof into two cases according to different values of $c\geq2$.

\vspace{1.5mm}
\noindent\textbf{Case 1.} $c=2$.
\vspace{1mm}

Then $G''= K_{2t-1} \vee (K_{n-2t}+K_{1})$. By (\ref{eq11}) and (\ref{eq12}), we have
$$\lambda_1(D(K_{2t-1} \vee (K_{n-2t}+K_{1})))\leq \lambda_1(D(G)),$$
where equality holds if and only if $G\cong K_{2t-1} \vee (K_{n-2t}+K_{1})$. By the assumption $\lambda_1(D(G))\leq\lambda_1(D(K_{2t-1} \vee (K_{n-2t}+K_{1}))),$ we have $\lambda_1(D(K_{2t-1} \vee (K_{n-2t}+K_{1})))=\lambda_1(D(G))$, and hence $G\cong K_{2t-1} \vee (K_{n-2t}+K_{1})$ (see Fig. \ref{fig2}). Take $S=V(K_{2t-1}).$ Then $$\frac{|S|}{c(K_{2t-1}\vee(K_{n-2t}+K_1)-S)}=\frac{2t-1}{2}<t,$$ and hence $\tau(K_{2t-1}\vee(K_{n-2t}+K_1)<t.$ This implies that $K_{2t-1} \vee (K_{n-2t}+K_{1})$ is not $t$-tough. So $G\cong K_{2t-1} \vee (K_{n-2t}+K_{1}).$

\vspace{1.5mm}
\noindent\textbf{Case 2.} $c\geq 3.$
\vspace{1mm}

\begin{figure}
\centering
\includegraphics[width=0.35\textwidth]{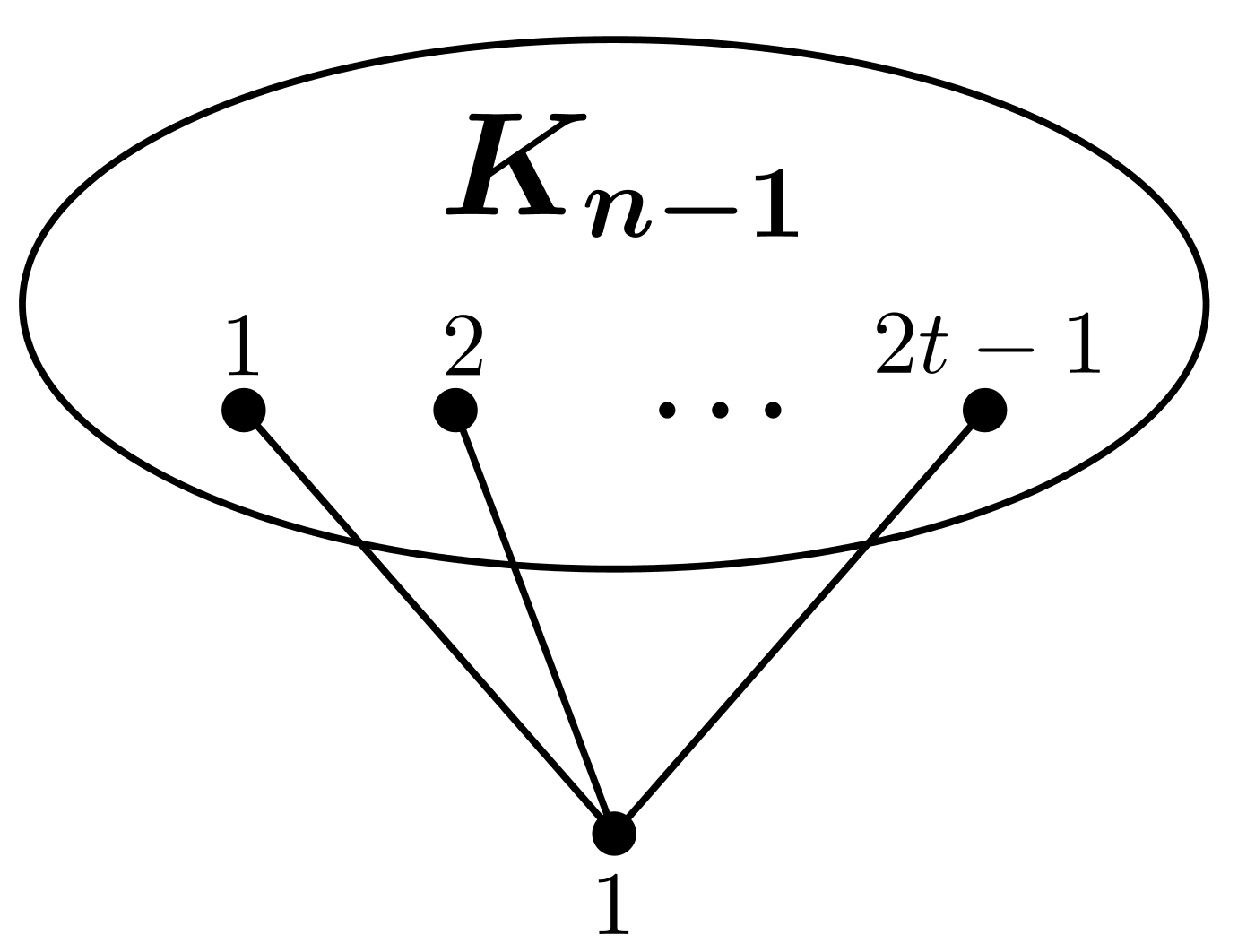}\\
\caption{Graph $K_{2t-1}\vee (K_{n-2t}+K_{1}).$
}\label{fig2}
\end{figure}

Recall that $G''=K_{tc-1}\vee(K_{n-(t+1)c+2}+(c-1)K_1).$ The distance matrix $D(G'')$ of $G''$ is
\[
\bordermatrix{
                  & c-1    & n-(t+1)c+2 & tc-1 \cr
\hfill c-1        & 2(J-I) &  2J        & J    \cr
\hfill n-(t+1)c+2 & 2J     & J-I        & J    \cr
\hfill tc-1       & J      & J          & J-I
}.
\]
By direct calculation, we have
\begin{eqnarray*}
W(G'')&=&\sum_{i<j}d_{ij}(G'')\\
&=&\frac{2[1+(c-2)](c-2)}{2}+2(c-1)[n-(t+1)c+2]+(c-1)(tc-1)\\
&&+\frac{[1+(n-c)](n-c)}{2}\\
&=&-\frac{2t+1}{2}c^2+\frac{2n+2t+3}{2}c+\frac{1}{2}n^2-\frac{3}{2}n-1.
\end{eqnarray*}
By Lemma \ref{le3}, we have
\begin{equation}\label{eq13}
\lambda_1(D(G''))\geq\frac{2W(G'')}{n}=\frac{-(2t+1)c^2+(2n+2t+3)c+n^2-3n-2}{n}.
\end{equation}
Define $\phi(c)=-(2t+1)c^2+(2n+2t+3)c+n^2-3n-2.$ It is easy to see that $n\geq (t+1)c-1$. Since $n\geq 4t^2+10t$ and $t\geq1,$ we obtain that
\begin{eqnarray*}
\phi(\frac{n+1}{t+1})-\phi(3)&=&\frac{n^2-(4t^2+9t+3)n+12t^3+26t^2+15t+2}{(t+1)^2}\\
&=&\frac{[n-(3t+2)][n-(4t^2+6t+1)]}{(t+1)^2}\\
&>&0.
\end{eqnarray*}
This implies that $\mathrm{min}_{3\leq c\leq\frac{n+1}{t+1}}\phi(c)=\phi(3)$. According to $n\geq 4t^2+10t, t\geq 1$ and (\ref{eq13}), we deduce that
\begin{eqnarray*}
\lambda_1(D(G''))\geq\frac{\phi(3)}{n}&=&\frac{n^2+3n-12t-2}{n}\\
&=&(n+2)+1-\frac{12t+2}{n}\\
&\geq&(n+2)+\frac{2(2t+1)(t-1)}{4t^2+10t}\\
&\geq&n+2.
\end{eqnarray*}
Define $G^*=K_{2t-1}\vee(K_{n-2t}+K_1).$ Then its distance matrix $D(G^*)$ is
\[
\bordermatrix{
                  & 2t-1    & n-2t  & 1    \cr
\hfill 2t-1       & J-I     & J     & J    \cr
\hfill n-2t       & J       & J-I   & 2J    \cr
\hfill 1          & J       & 2J    & O
}.
\]
By simple calculation, we have
\begin{eqnarray*}
W^{(2)}(G^*)=\sum_{1\leq i<j\leq n}d_{ij}^2(G^*)&=& \frac{[1+(2t-2)](2t-2)}{2}+(2t-1)(n-2t+1)\\
&&+\frac{[1+(n-2t-1)](n-2t-1)}{2}+4(n-2t)\\
&=&\frac{1}{2}n^2+\frac{5}{2}n-6t.
\end{eqnarray*}
Note that $t\geq 1$ and $n\geq 4t^2+10t.$ By Lemma \ref{le7}, we have
\begin{eqnarray*}
\lambda_1(D(G^*))\leq\sqrt{\frac{2(n-1)W^{(2)}(G^*)}{n}}&=&\sqrt{\frac{-12(n-1)t+n^3+4n^2-5n}{n}}\\
&\leq&\sqrt{\frac{n^3+4n^2-17n+12}{n}}\\
&=&\sqrt{(n+2)^2+\frac{12}{n}-21}\\
&<&n+2.
\end{eqnarray*}
Therefore, $\lambda_1(D(G^*))<n+2\leq \lambda_1(D(G'')).$ Combining this with (\ref{eq11}) and (\ref{eq12}), we have
$$\lambda_1(D(G^*))<\lambda_1(D(G''))\leq\lambda_1(D(G'))\leq\lambda_1(D(G)),$$
a contradiction.

(ii) When $1/t$ is a positive number, we have $c\geq \frac{s}{t}+1.$ It is obvious that $G$ is a spanning subgraph of $\hat{G}=K_{s}\vee(K_{n_1}+K_{n_2}+\cdots+K_{n_{\frac{s}{t}+1}})$ for $n_1\geq n_2\geq\cdots\geq n_{{\frac{s}{t}+1}}\geq 1$ and $\sum_{i=1}^{{\frac{s}{t}+1}}n_i=n-s.$ According to Lemma \ref{le1}, we obtain that
\begin{equation}\label{eq14}
\lambda_1(D(\hat{G)})\leq \lambda_1(D(G)),
\end{equation}
with equality holding if and only if $G\cong \hat{G}$.
Let $\tilde{G}=K_{s}\vee(K_{n-s+\frac{s}{t}}+\frac{s}{t}K_1).$ By Lemma \ref{le4}, we have
\begin{equation}\label{eq15}
\lambda_1(D(\tilde{G}))\leq \lambda_1(D(\hat{G})),
\end{equation}
where equality holds if and only if $\hat{G}\cong \tilde{G}.$
Next we consider the following two cases depending on the value of $s\geq1$.

\vspace{1.5mm}
\noindent\textbf{Case 1.} $s=1$.
\vspace{1mm}

Then $\tilde{G}= K_{1} \vee (K_{n-1-\frac{1}{t}}+\frac{1}{t}K_{1})$. By (\ref{eq14}) and (\ref{eq15}), we conclude that
$$\lambda_1(D(K_{1} \vee (K_{n-1-\frac{1}{t}}+\frac{1}{t}K_{1})))\leq \lambda_1(D(G)),$$
with equality holding if and only if $G\cong K_{1} \vee (K_{n-1-\frac{1}{t}}+\frac{1}{t}K_{1})$. By the assumption $\lambda_1(D(G))\leq\lambda_1(D(K_{1} \vee (K_{n-1-\frac{1}{t}}+\frac{1}{t}K_{1}))),$ we have $\lambda_1(D(G))=\lambda_1(D(K_{1} \vee (K_{n-1-\frac{1}{t}}+\frac{1}{t}K_{1})))$, and hence $G\cong K_{1} \vee (K_{n-1-\frac{1}{t}}+\frac{1}{t}K_{1})$ (see Fig. \ref{fig3}). Take $S=V(K_{1}).$ Then $$\frac{|S|}{c(K_{1} \vee (K_{n-1-\frac{1}{t}}+\frac{1}{t}K_{1})-S)}=\frac{1}{1+\frac{1}{t}}<t,$$ and hence $\tau(K_{1} \vee (K_{n-1-\frac{1}{t}}+\frac{1}{t}K_{1}))<t.$ This implies that $K_{2t-1} \vee (K_{n-2t}+K_{1})$ is not $t$-tough. So $G\cong K_{2t-1} \vee (K_{n-2t}+K_{1}).$

\vspace{1.5mm}
\noindent\textbf{Case 2.} $s\geq 2.$
\vspace{1mm}

\begin{figure}
\centering
\includegraphics[width=0.35\textwidth]{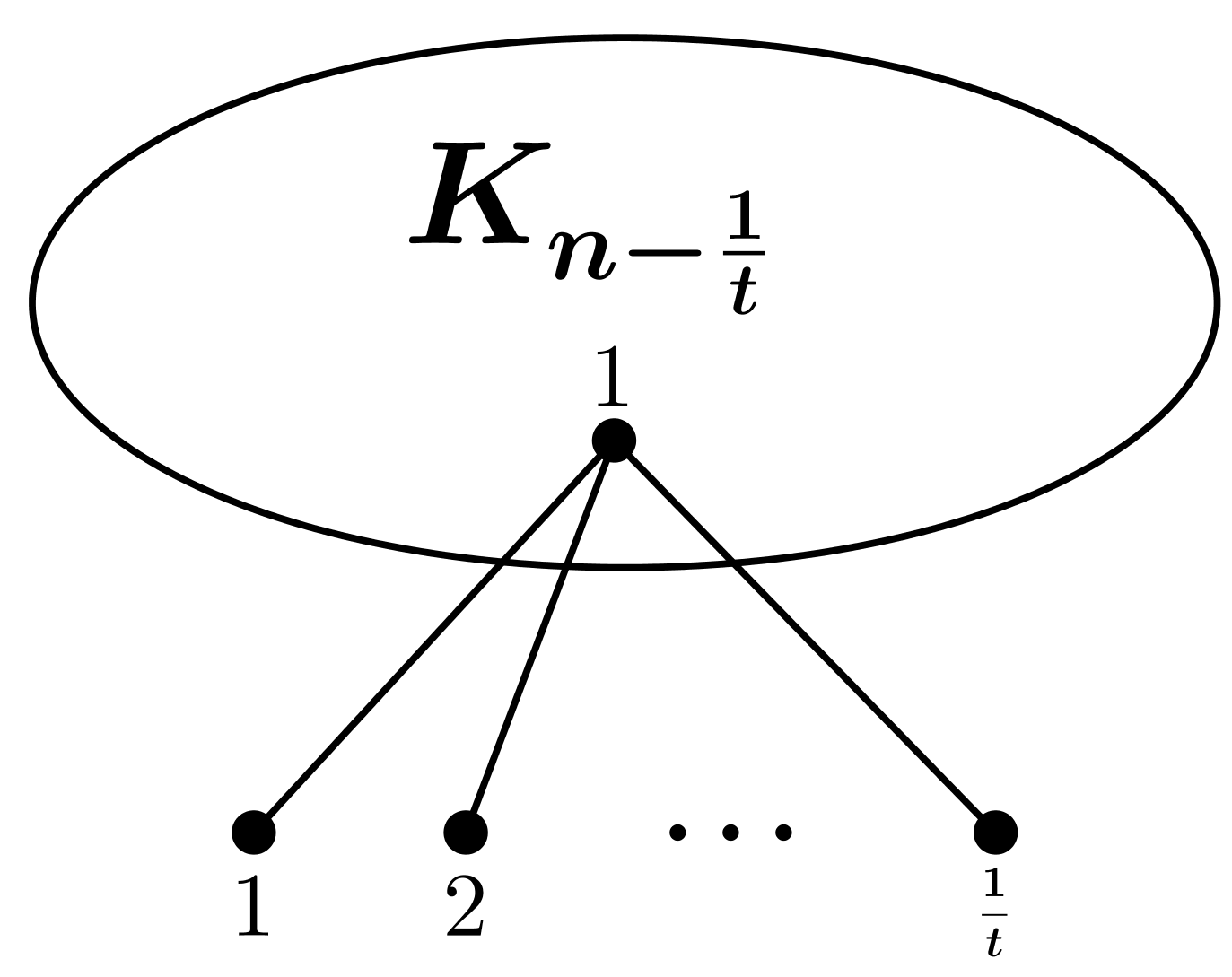}\\
\caption{Graph $K_{1} \vee (K_{n-\frac{1}{t}-1}+\frac{1}{t}K_{1}).$
}\label{fig3}
\end{figure}

Recall that $\tilde{G}=K_{s}\vee(K_{n-s+\frac{s}{t}}+\frac{s}{t}K_1).$ Notice that $D(\tilde{G})$ has the equitable quotient matrix
$$R_{t,s}=\left(
\begin{array}{ccc}
2(\frac{s}{t}-1) &2(n-s-\frac{s}{t}) & s\\
\frac{2s}{t}     & n-s-\frac{s}{t}-1 & s\\
\frac{s}{t}      & n-s-\frac{s}{t}   & s-1
\end{array}
\right).
$$
By simple calculation, the characteristic polynomial of $R_{t,s}$ is
\begin{eqnarray*}
P(R_{t,s},x)&=&x^3-\frac{tn+s-4t}{t}x^2-\frac{3t^2n+2stn-5t^2-3s^2t+st-2s^2}{t^2}x\\
&&+\frac{-2t^2n+s^2tn-2stn+2t^2-s^3t+3s^2t-s^3+2s^2}{t^2}.
\end{eqnarray*}
Let $G^{**}=K_1\vee(K_{n-1-\frac{1}{t}}+\frac{1}{t}K_{1}).$
Note that $D(G^{**})$ has the equitable quotient matrix $R_{t},$ which is obtained by taking $s=1$ in $R_{t,s}.$
Then
$P(R_t,x)-P(R_{t,s},x)=\frac{s-1}{t^2}\psi(x),$
where $$\psi(x)=tx^2+(2tn-3st-2t-2s-2)x-stn+tn+s^2t-2st-2t+s^2-s-1.$$
It is clear that $n\geq s+\frac{s}{t}+1.$ Then $2\leq s\leq \frac{n-1}{1+\frac{1}{t}},$ and hence the symmetry axis of $\psi(x)$ is $$x=-n+\frac{s}{t}+\frac{1}{t}+\frac{3}{2}s+1<n-\frac{1}{t}-1.$$
This implies that $\psi(x)$ is monotonically increasing with respect to $x\in [n-\frac{1}{t}-1, +\infty)$. Since $n\geq 2t+\frac{9}{2 t}+\frac{9}{2}$ and $2\leq s\leq \frac{n-1}{1+\frac{1}{t}},$ we have
\begin{eqnarray*}
\psi(x)&\geq& \psi(n-\frac{1}{t}-1)\\
&=& (t+1)s^2-(4tn+2n-t-\frac{2}{t}-4)s+3tn^2-5tn-6n+t+\frac{3}{t}+5\\
&\geq& (t+1)(\frac{n-1}{1+\frac{1}{t}})^2-(4tn+2n-t-\frac{2}{t}-4)(\frac{n-1}{1+\frac{1}{t}})+3tn^2-5tn-6n+t+\frac{3}{t}+5\\
&=& \frac{1}{t(t+1)}[t^2n^2-(2t^3+5t^2+4t)n+t^3+2t^2+6t+3]\\
&\geq& \frac{1}{t(t+1)}[t^2(2t+\frac{9}{2}t+\frac{9}{2})^2-(2t^3+5t^2+4t)(2t+\frac{9}{2}t+\frac{9}{2})+t^3+2t^2+6t+3]\\
&=& \frac{3t^2+24t+84}{4t(t+1)}\\
&>& 0.
\end{eqnarray*}
Combining this with $s\geq 2,$ we deduce that $P(R_t,x)-P(R_{t,s},x)=\frac{s-1}{t^2}\psi(x)>0$ for $x\in [n-\frac{1}{t}-1, +\infty).$ Since $K_{n-\frac{1}{t}}$ is a proper subgraph of $G^{**},$ we obtain that $\lambda_1(D(G^{**}))>\lambda_1(D(K_{n-\frac{1}{t}}))=n-\frac{1}{t}-1.$
Hence $\lambda_1(R_t)<\lambda_1(R_{t,s}).$
Combining this with Lemma \ref{le2}, we have $\lambda_1(D(G^{**}))<\lambda_1(D(\tilde{G})).$
By (\ref{eq14}) and (\ref{eq15}), then
$$\lambda_1(D(G^{**}))<\lambda_1(D(\tilde{G}))\leq\lambda_1(D(\hat{G}))\leq\lambda_1(D(G)),$$
which contradicts the assumption. \hspace*{\fill}$\Box$

\vspace{5mm}
\noindent
\vspace{3mm}




\end{document}